\documentclass[9pt,shortpaper,twoside,web]{IEEEtran}
\usepackage{cite}

\ifCLASSINFOpdf
  \usepackage[pdftex]{graphicx}
  \usepackage{epstopdf}
  \usepackage{subfigure}
  \graphicspath{{./}{./figures/}}
\else
\fi
\usepackage{xcolor}

\usepackage{amsmath}
\usepackage{amssymb}

\def\R{\mathbb{R}}
\def\N{\mathbb{N}}

\newtheorem{nnassumption}{\bf Assumption}

\newtheorem{nntheorem}{\bf Theorem}
\newenvironment{theorem}{\begin{nntheorem}\it}{\end{nntheorem}}
\newtheorem{nncorollary}{\bf Corollary}

\newtheorem{nndefinition}{\bf Definition}

\newtheorem{nnproposition}{\bf Proposition}

\newtheorem{nnproblem}{\bf Problem}

\newtheorem{nnlemma}{\bf Lemma}
\newenvironment{lemma}{\begin{nnlemma}\it}{\end{nnlemma}}
\newtheorem{nnremark}{\bf Remark}
\newenvironment{remark}{\begin{nnremark} \rm }{\hfill \hspace*{1pt}\hfill $\circ$\end{nnremark}}
\newtheorem{nnexample}{\bf Example}

\newenvironment{proof}{{\bf Proof.}}{\hfill \hspace*{1pt}\hfill $\Box$}

\allowdisplaybreaks

\begin{document}
%
\title{Finite-dimensional observer-based PI regulation control of a reaction-diffusion equation}
%
%
%

\author{Hugo~Lhachemi and Christophe~Prieur
\thanks{Hugo Lhachemi is with Universit{\'e} Paris-Saclay, CNRS, CentraleSup{\'e}lec, Laboratoire des signaux et syst{\`e}mes, 91190, Gif-sur-Yvette, France (e-mail: hugo.lhachemi@centralesupelec.fr). His work has been partially supported by ANR PIA funding: ANR-20-IDEES-0002.
\newline\indent Christophe Prieur is with Universit{\'e} Grenoble Alpes, CNRS, Grenoble-INP, GIPSA-lab, F-38000, Grenoble, France (e-mail: christophe.prieur@gipsa-lab.fr). His work has been partially supported by MIAI@Grenoble Alpes (ANR- 19-P3IA-0003)}
}

%
%

\markboth{Manuscript submitted to IEEE Transaction on Automatic Control}%
{Lhachemi \MakeLowercase{\textit{et al.}}}
%



\maketitle

\begin{abstract}
This paper investigates the output feedback setpoint regulation control of a reaction-diffusion equation by means of boundary control. The considered reaction-diffusion plant may be open-loop unstable. The proposed control strategy consists of the coupling of a finite-dimensional observer and a PI controller in order to achieve the boundary setpoint regulation control of various system outputs such as the Dirichlet and Neumann traces. In this context, it is shown that the order of the finite-dimensional observer can always be selected large enough, with explicit criterion, to achieve both the stabilization of the plant and the setpoint regulation of the system output.
\end{abstract}


%
\IEEEpeerreviewmaketitle

\section{Introduction}\label{sec: Introduction}
The problem of controlling the output of a system so as to achieve asymptotic tracking of prescribed trajectories is one of the most fundamental problems in control theory. In the general context of finite-dimensional linear time-invariant (LTI) control systems, the problem of setpoint regulation control is very classical and has been widely investigated. One possible way to solve this problem is based on the augmentation of the state-space representation of the plant with an integral component of the tracking error and the use of the separation principle by exploiting separately a Luenberger observer (which allows the estimation of the state based on the measure only) and a stabilizing full-state feedback (see, e.g., \cite{hespanha2018linear}). Even if this approach has reached a very high level of maturity for finite-dimensional systems, its possible extension to infinite-dimensional systems, as those considered in this paper, is still an open problem.

Infinite-dimensional systems emerge in many practical applications due to the occurrence of delays, reaction-diffusion dynamics, or even flexible behavior (see, e.g., \cite{meurer2012control,morris2020,krstic2008boundary} for introductory textbooks on dedicated control theory for infinite-dimensional systems). While many efficient control design methods have been reported for the stabilization of distributed parameter systems, very few have been extended to the problem of output regulation. The main reason is that all techniques that have been developed for finite-dimensional LTI systems cannot be easily generalized to infinite-dimensional plants. For instance, the frequency domain approach has been generalized to the infinite-dimensional setting, but it requires to deal with an infinite number of poles, yielding an infinite-dimensional pole allocation problem. The state-space approach is followed in this work.

We propose, for the first time, an output feedback control design procedure to achieve the setpoint regulation control of reaction-diffusions system by means of a finite-dimensional observer coupled with a PI controller. The considered reaction-diffusion plant, which might be unstable, is modeled by a Sturm-Liouville operator as those classical introduced in the context of parabolic partial differential equation (PDE). The case of PI regulation of this system by means of a state feedback was reported in~\cite{lhachemi2020pi} (see also~\cite{pohjolainen1982robust,xu1995robust,dos2008boundary,trinh2017design,terrand2018regulation,barreau2019practical,terrand2019adding,coron2019pi,lhachemi2020pi,lhachemi2021pi} for various approaches about PI control design for different types of PDEs). Here we go beyond by designing an output feedback PI control strategy. Even if the proposed procedure also applies to bounded control inputs and bounded observations, we focus the presentation on boundary controls and boundary measurements. This is because these configurations are the most interesting for practical applications and also the most challenging since they involve unbounded control and observation operators (see, e.g., \cite{curtain2020introduction} for further explanations). We study several cases for the input-to-output map, covering Dirichlet control inputs (easily extendable to Neumann control inputs as discussed in conclusion) along with Dirichlet and/or Neumann to-be-regulated outputs and measured outputs. We also show that our procedure can be used to regulate a system output that is distinct of the measured one. Therefore, our approach gives a complete framework to study every associated input-to-output maps. 

The proposed control design strategy consists of an adequate integral component coupled with a finite-dimensional observer. The design of finite-dimensional observer-based controllers for distributed parameter plants is challenging due to the fact that the separation principle, that is classically used for finite-dimensional systems, does not apply for infinite-dimensional systems~\cite{sakawa1983feedback,curtain1982finite,balas1988finite,harkort2011finite}. Taking advantage of spectral reduction approaches~\cite{russell1978controllability,coron2004global} and using the control architecture initially reported in~\cite{sakawa1983feedback}, a LMI-based procedure for solving this stabilization problem for reaction-diffusion PDEs was reported in~\cite{katz2020constructive} in the case were the either control or observation operator is bounded. This approach was extended in~\cite{lhachemi2020finite} to the case were both control and observation operators are unbounded, including both Dirichlet and Neumann settings. The present work, taking advantage of~\cite{lhachemi2020finite}, goes beyond the simple problem of closed-loop stabilization by embracing the issue of output setpoint regulation control. Since the designed observer only estimates a finite number of modes of the infinite-dimensional system, there is an inherent mismatch between the actually measured system output and its estimation as soon as the output is to be regulated to a non-zero value. Hence, one of the main challenges is to account for this mismatch in the dynamics of the observer and then in the subsequent stability analysis. An other challenge is to couple this finite-dimensional observer with a suitable integral component, inspired by the one described in~\cite{lhachemi2020pi} for a state-feedback, in an output feedback setting. Our approach is based on Lyapunov direct methods and the main results take the form of explicit sufficient conditions ensuring the both stability and setpoint regulation control of the closed-loop plant. We assess that these conditions are always feasible provided the order of the observer is selected large enough. Therefore, we show in a constructive manner that the setpoint regulation control of reaction-diffusion PDEs can always be achieved by the coupling of a PI and a finite-dimensional observer.

The paper is organized by considering successively different input-output maps for the reaction-diffusion equation depending on the selected boundary measured output, the to-be-regulated output, and the control input. After recalling classical notations and properties for the Sturm-Liouville operators in Section~\ref{sec: preliminaries}, the case of a Dirichlet observation and a Dirichlet control input is considered in Section~\ref{sec: Dirichlet measurement and regulation control}. Then the case of a Neumann measurement and a Dirichlet control input is considered in Section~\ref{sec: Neuman}. While the to-be-regulated output and the measured output are the same in the two latter sections, a crossed configuration is considered in Section~\ref{sec: crossed configuration}. The regulation problem is solved for a Dirichlet measured output, a Neumann to-be-regulated output, and a Dirichlet control input. This final result completes the picture and gives a full study of the different cases for the input-to-output map of the considered class of distributed parameter systems. Some numerical simulations are given in Section \ref{sec: num} for this final result. Section~\ref{sec: conclusion} collects some concluding remarks.

\section{Notation and properties}\label{sec: preliminaries}

Spaces $\R^n$ are endowed with the Euclidean norm denoted by $\Vert\cdot\Vert$. The associated induced norms of matrices are also denoted by $\Vert\cdot\Vert$. Given two vectors $X$ and $Y$, $ \mathrm{col} (X,Y)$ denotes the vector $[X^\top,Y^\top]^\top$. $L^2(0,1)$ stands for the space of square integrable functions on $(0,1)$ and is endowed with the inner product $\langle f , g \rangle = \int_0^1 f(x) g(x) \,\mathrm{d}x$ with associated norm denoted by $\Vert \cdot \Vert_{L^2}$. For an integer $m \geq 1$, the Sobolev space of order $m$ is denoted by $H^m(0,1)$ and is endowed with its usual norm denoted by $\Vert \cdot \Vert_{H^m}$. For a symmetric matrix $P \in\R^{n \times n}$, $P \succeq 0$ (resp. $P \succ 0$) means that $P$ is positive semi-definite (resp. positive definite).

Let $p \in \mathcal{C}^1([0,1])$ and $q \in \mathcal{C}^0([0,1])$ with $p > 0$ and $q \geq 0$. Let the Sturm-Liouville operator $\mathcal{A} : D(\mathcal{A}) \subset L^2(0,1) \rightarrow L^2(0,1)$ be defined by $\mathcal{A}f = - (pf')' + q f$ on the domain $D(\mathcal{A}) \subset L^2(0,1)$ given by either $D(\mathcal{A}) = \{ f \in H^2(0,1) \,:\, f(0)=f(1)=0 \}$ or $D(\mathcal{A}) = \{ f \in H^2(0,1) \,:\, f'(0)=f(1)=0 \}$. The eigenvalues $\lambda_n$, $n \geq 1$, of $\mathcal{A}$ are simple, non negative, and form an increasing sequence with $\lambda_n \rightarrow + \infty$ as $n \rightarrow + \infty$. Moreover, the associated unit eigenvectors $\phi_n \in L^2(0,1)$ form a Hilbert basis. We also have $D(\mathcal{A}) = \{ f \in L^2(0,1) \,:\, \sum_{n\geq 1} \vert \lambda_n \vert ^2 \vert \left< f , \phi_n \right> \vert^2 \}$ and $\mathcal{A}f = \sum_{n \geq 1} \lambda_n \left< f , \phi_n \right> \phi_n$.

Let $p_*,p^*,q^* \in \R$ be such that $0 < p_* \leq p(x) \leq p^*$ and $0 \leq q(x) \leq q^*$ for all $x \in [0,1]$, then it holds~\cite{orlov2017general}:
\begin{equation}\label{eq: estimation lambda_n}
0 \leq \pi^2 (n-1)^2 p_* \leq \lambda_n \leq \pi^2 n^2 p^* + q^*
\end{equation}
for all $n \geq 1$. Assuming further than $p \in \mathcal{C}^2([0,1])$, we have for any $x \in \{0,1\}$ that $\phi_n (x) = O(1)$ and $\phi_n' (x) = O(\sqrt{\lambda_n})$ as $n \rightarrow + \infty$ \cite{orlov2017general}. Finally, one can check that, for all $f \in D(\mathcal{A})$,
\begin{align}
\sum_{n \geq 1} \lambda_n \left< f , \phi_n \right>^2
& = \left< \mathcal{A}f , f \right>
= \int_0^1 p (f')^2 + q f^2 \,\mathrm{d}x . \label{eq: inner product Af and f}
\end{align}
Moreover, for any $f \in D(\mathcal{A})$, we have $f(x) = \sum_{n \geq 1} \left< f , \phi_n \right> \phi_n(x)$ and $f'(x) = \sum_{n \geq 1} \left< f , \phi_n \right> \phi_n'(x)$ for all $x \in [0,1]$ (see, e.g., \cite{lhachemi2020finite}).

\section{Dirichlet measurement and regulation control}\label{sec: Dirichlet measurement and regulation control}

We consider the reaction-diffusion system with Dirichlet boundary observation described for $t > 0$ and $x \in (0,1)$ by
\begin{subequations}\label{eq: dirichlet boundary measurement - RD system}
\begin{align}
z_t(t,x) & = \left( p(x) z_x(t,x) \right)_x + (q_c - q(x)) z(t,x) \label{eq: dirichlet boundary measurement - RD system - 1} \\
z_x(t,0) & = 0 , \quad z(t,1) = u(t) \label{eq: dirichlet boundary measurement - RD system - 2} \\
z(0,x) & = z_0(x) \label{eq: dirichlet boundary measurement - RD system - 3} \\
y(t) & = z(t,0) \label{eq: dirichlet boundary measurement - RD system - 4}
\end{align}
\end{subequations}
in the case $p \in \mathcal{C}^2([0,1])$. Here $q_c \in\R$ is a constant, $u(t) \in\R$ is the command input, $y(t) \in\R$ is the measurement, $z_0 \in L^2(0,1)$ is the initial condition, and $z(t,\cdot) \in L^2(0,1)$ is the state. The control design objective is to design a finite-dimensional observer-based controller to achieve both stabilization and setpoint regulation control of $y(t)$ to some prescribed reference signal $r(t)$. By setpoint tracking, we mean that our objective is to ensure that $y(t) \rightarrow r_e$ when $t \rightarrow +\infty$ as soon as $r(t) \rightarrow r_e$ when $t \rightarrow +\infty$ for any arbitrarily given $r_e \in\R$.

\subsection{Spectral reduction}

As classically done in the context of boundary control systems (see \cite[Sec.~3.3]{curtain2020introduction} for details), we start by transforming the non-homogeneous problem (\ref{eq: dirichlet boundary measurement - RD system}) into an equivalent homogeneous one by introducing the change of variable:
\begin{equation}\label{eq: change of variable}
w(t,x) = z(t,x) - x^2 u(t) 
\end{equation}
that gives the equivalent representation:
\begin{subequations}\label{eq: homogeneous RD system}
\begin{align}
& w_t(t,x) = ( p(x) w_x(t,x) )_x + (q_c - q(x)) w(t,x) \nonumber \\
& \phantom{w_t(t,x) =}\; + a(x) u(t) + b(x) \dot{u}(t) \label{eq: homogeneous RD system - 1} \\
& w_x(t,0) = 0 , \quad w(t,1) = 0 \label{eq: homogeneous RD system - 2} \\
& w(0,x) = w_0(x) \label{eq: homogeneous RD system - 3} \\
& \tilde{y}(t) = w(t,0) \label{eq: homogeneous RD system - 4}
\end{align}
\end{subequations}
with $a,b \in L^2(0,1)$ defined by $a(x) = 2p(x) + 2xp'(x) + (q_c-q(x))x^2$ and $b(x) = -x^2$, respectively, $\tilde{y}(t) = y(t)$, and $w_0(x) = z_0(x) - x^2 u(0)$. Introducing the auxiliary command input $v(t) = \dot{u}(t)$, we infer that
\begin{subequations}\label{eq: homogeneous RD system - abstract form}
\begin{align}
\dot{u}(t) & = v(t) \label{eq: homogeneous RD system - abstract form - auxiliary input v} \\
\dfrac{\mathrm{d} w}{\mathrm{d} t}(t,\cdot) & = - \mathcal{A} w(t,\cdot) + q_c w(t,\cdot) + a u(t) + b v(t)
\end{align}
\end{subequations}
with $D(\mathcal{A}) = \{ f \in H^2(0,1) \,:\, f'(0)=f(1)=0 \}$. We introduce the coefficients of projection $w_n(t) = \left< w(t,\cdot) , \phi_n \right>$, $a_n = \left< a , \phi_n \right>$, and $b_n = \left< b , \phi_n \right>$. Considering classical solutions associated with any $z_0 \in H^2(0,1)$ and any $u(0) \in \R$ such that $z_0'(0)=0$ and $z_0(1) = u(0)$ (their existence for the upcoming closed-loop dynamics is an immediate consequence of \cite[Chap.~6, Thm.~1.7]{pazy2012semigroups}), we have $w(t,\cdot) \in D(\mathcal{A})$ for all $t \geq 0$ and we infer that
\begin{subequations}\label{eq: homogeneous RD system - spectral reduction}
\begin{align}
\dot{u}(t) & = v(t) \label{eq: homogeneous RD system - spectral reduction - 1} \\
\dot{w}_n(t) & = ( -\lambda_n + q_c ) w_n(t) + a_n u(t) + b_n v(t) \, , \quad n \geq 1 \label{eq: homogeneous RD system - spectral reduction - 2} \\
\tilde{y}(t) & = \sum_{i \geq 1} \phi_i(0) w_i(t) \label{eq: homogeneous RD system - spectral reduction - 3}
\end{align}
\end{subequations}

\subsection{Control design}

Let $\delta > 0$ be the desired exponential decay rate for the setpoint regulation. We fix $N_0 \geq 1$ so that $- \lambda_n + q_c < -\delta < 0$ for all $n \geq N_0 +1$. The integer $N_0$, which is definitely fixed for the rest of the control design procedure, can be interpreted as the number of modes that will be ``actively'' modified by the feedback. We now introduce an arbitrary integer $N \geq N_0 + 1$ which will be further constrained later. Inspired by~\cite{sakawa1983feedback}, we design as in~\cite{lhachemi2020finite} an observer to estimate the $N$ first modes of the plant while the state-feedback is performed on the $N_0$ first modes of the plant. In this framework, the estimation of the modes ranging from $N_0 + 1$ to $N$ will solely be used to improve the estimate of the system output (see (\ref{eq: dirichlet boundary measurement - observer dynamics - 1})).

Introducing $W^{N_0}(t) = \begin{bmatrix} w_{1}(t) & \ldots & w_{N_0}(t) \end{bmatrix}^\top$, $A_0 = \mathrm{diag}(- \lambda_{1} + q_c , \ldots , - \lambda_{N_0} + q_c)$, $B_{0,a} = \begin{bmatrix} a_1 & \ldots & a_{N_0} \end{bmatrix}^\top$, and $B_{0,b} = \begin{bmatrix} b_1 & \ldots & b_{N_0} \end{bmatrix}^\top$, we have
\begin{equation}\label{eq: W^N0 dynamics - 1}
\dot{W}^{N_0}(t) = A_0 W^{N_0}(t) + B_{0,a} u(t) + B_{0,b} v(t) .
\end{equation}
Our objective is to introduce an integral component to achieve the setpoint regulation control of the system output $y(t)$. To do so, we first consider the following classical integral component: $\dot{z}_i(t) = y(t) - r(t) = \sum_{n \geq 1} \phi_n(0) w_n(t) - r(t)$. This $z_i$-dynamics, which involves all the modes $w_n$ for $n \geq 1$, cannot be embedded into the reduced model (\ref{eq: W^N0 dynamics - 1}) that only involves the modes $w_n$ for $1 \leq n \leq N_0$. To circumvent this issue, we follow the idea developed in~\cite{lhachemi2020pi} by introducing $\xi_p(t) = z_i(t) - \sum_{n \geq N_0 +1} \frac{\phi_n(0)}{-\lambda_n +q_c} w_n(t)$ whose time derivative is given by $\dot{\xi}_p(t)
= \sum_{n = 1}^{N_0} \phi_n(0) w_n(t) + \alpha_0 u(t) + \beta_0 v(t) - r(t)$ with
\begin{align}\label{eq: def alpha_0 and beta_0}
\alpha_0 = - \sum_{n \geq N_0 + 1} \frac{a_n \phi_n(0)}{-\lambda_n + q_c} , \quad \beta_0 = - \sum_{n \geq N_0 + 1} \frac{b_n \phi_n(0)}{-\lambda_n + q_c} .
\end{align}
The main benefit is that the $\xi_p$-dynamics only involves the modes $w_n$ for $1 \leq n \leq N_0$ while achieving the same equilibrium condition than the $z_i$-dynamics. However, in this work and in sharp contrast with the state-feedback setting of~\cite{lhachemi2020pi}, the modes $w_n$ are not measured. Hence we need to replace them in the dynamics of the integral component by their estimated version $\hat{w}_n$ which will be described below. Hence, the employed integral component is described by:
\begin{align}\label{eq: dirichlet boundary measurement - integral component}
\dot{\xi}(t)
& = \sum_{n = 1}^{N_0} \phi_n(0) \hat{w}_n(t) + \alpha_0 u(t) + \beta_0 v(t) - r(t) .
\end{align}
We now define for $1 \leq n \leq N$ the observer dynamics:
\begin{align}
\dot{\hat{w}}_n (t)
& = ( -\lambda_n + q_c ) \hat{w}_n(t) + a_n u(t) + b_n v(t) \nonumber \\
& \phantom{=}\; - l_n \left( \sum_{i=1}^N \phi_i(0) \hat{w}_i(t) - \alpha_1 u(t) - \tilde{y}(t) \right) \label{eq: dirichlet boundary measurement - observer dynamics - 1}
\end{align}
with
\begin{equation}\label{eq: dirichlet boundary measurement - def alpha_1}
\alpha_1 = \sum_{n \geq N +1} \dfrac{a_n \phi_n(0)}{-\lambda_n + q_c}
\end{equation}
and where $l_n \in\R$ are the observer gains. We set $l_n = 0$ for $N_0+1 \leq n \leq N$. Compared to the stabilization problem studied in~\cite{lhachemi2020finite}, we introduce the additional term $-\alpha_1 u(t)$ in the observer dynamics (\ref{eq: dirichlet boundary measurement - observer dynamics - 1}). This term is added to compensate the inherent steady-state mismatch between the actually measured system output $\tilde{y}(t)$ and its estimation $\sum_{i=1}^N \phi_i(0) \hat{w}_i(t)$, obtained from the observer that estimates the only $N$ first modes of the plant, as soon as the output is to be regulated to a non-zero value. Note that this latter estimate of the output improves as the dimension of the observer $N$ increases.

We define for $1 \leq n \leq N$ the observation error as $e_n(t) = w_n(t) - \hat{w}_n(t)$. Hence we have
\begin{align}
& \dot{\hat{w}}_n (t) = ( -\lambda_n + q_c ) \hat{w}_n(t) + a_n u(t) + b_n v(t) + l_n \sum_{i=1}^{N_0} \phi_i(0) e_i(t) \nonumber \\
& + l_n \sum_{i=N_0+1}^{N} \dfrac{\phi_i(0)}{\sqrt{\lambda_i}} \tilde{e}_i(t) + l_n \alpha_1 u(t) + l_n \zeta(t) \label{eq: dirichlet boundary measurement - observer dynamics - 2}
\end{align}
with $\tilde{e}_n(t) = \sqrt{\lambda_n} e_n(t)$ and $\zeta(t) = \sum_{n \geq N+1} \phi_n(0) w_n(t)$. Hence, introducing $\hat{W}^{N_0}(t) = \begin{bmatrix} \hat{w}_{1}(t) & \ldots & \hat{w}_{N_0}(t) \end{bmatrix}^\top$, $E^{N_0}(t) = \begin{bmatrix} e_{1}(t) & \ldots & e_{N_0}(t) \end{bmatrix}^\top$, $\tilde{E}^{N-N_0}(t) = \begin{bmatrix} \tilde{e}_{N_0 +1} & \ldots & \tilde{e}_N \end{bmatrix}^\top$, $C_0 = \begin{bmatrix} \phi_1(0) & \ldots & \phi_{N_0}(0) \end{bmatrix}$, $C_1 = \begin{bmatrix} \frac{\phi_{N_0 +1}(0)}{\sqrt{\lambda_{N_0 + 1}}} & \ldots & \frac{\phi_{N}(0)}{\sqrt{\lambda_{N}}} \end{bmatrix}$, and $L = \begin{bmatrix} l_1 & \ldots & l_{N_0} \end{bmatrix}^\top$, we obtain that
\begin{align}
\dot{\hat{W}}^{N_0}(t) & = A_0 \hat{W}^{N_0}(t) + B_{0,a} u(t) + B_{0,b} v(t) + L C_0 E^{N_0}(t) \nonumber \\
& \phantom{=}\; + L C_1 \tilde{E}^{N-N_0}(t) + \alpha_1 L u(t) + L \zeta(t) \label{eq: hat_W^N0 dynamics - 2} .
\end{align}
With
\begin{equation}\label{eq: def hat W^N_0_a(t)}
\hat{W}^{N_0}_a(t) = \mathrm{col} (u(t),\hat{W}^{N_0}(t),\xi(t)) ,
\end{equation}
$\tilde{L} = \mathrm{col}(0,L,0)$, and defining
\begin{equation}\label{eq: def matrices A1 B1 and Br}
A_1 = \begin{bmatrix} 0 & 0 & 0 \\ B_{0,a} & A_0 & 0 \\ \alpha_0 & C_0 & 0 \end{bmatrix} , \quad
B_{1} = \begin{bmatrix} 1 \\ B_{0,b} \\ \beta_0 \end{bmatrix} , \quad
B_{r} = \begin{bmatrix} 0 \\ 0 \\ 1 \end{bmatrix} ,
\end{equation}
we deduce that
\begin{align}
\dot{\hat{W}}_a^{N_0}(t) & = A_1 \hat{W}_a^{N_0}(t) + B_{1} v(t) - B_{r} r(t) + \tilde{L} C_0 E^{N_0}(t) \nonumber \\
& \phantom{=}\; + \tilde{L} C_1 \tilde{E}^{N-N_0}(t) + \alpha_1 \tilde{L} u(t) + \tilde{L} \zeta(t) . \label{eq: hat_W_a^N0 dynamics - 1}
\end{align}
Setting the auxiliary command input as
\begin{equation}\label{eq: v - state feedback}
v(t) = K \hat{W}_a^{N_0}(t) ,
\end{equation}
and defining
\begin{equation}\label{eq: def Acl}
A_\mathrm{cl}(\alpha_1) = A_1 + B_1 K + \alpha_1 \tilde{L} \begin{bmatrix} 1 & 0 & 0 \end{bmatrix} ,
\end{equation}
we obtain that
\begin{align}
\dot{\hat{W}}_a^{N_0}(t) & = A_\mathrm{cl}(\alpha_1) \hat{W}_a^{N_0}(t) - B_{r} r(t) \nonumber \\
& \phantom{=}\; + \tilde{L} C_0 E^{N_0}(t) + \tilde{L} C_1 \tilde{E}^{N-N_0}(t) + \tilde{L} \zeta(t) \label{eq: hat_W_a^N0 dynamics - 2}
\end{align}
and, from (\ref{eq: W^N0 dynamics - 1}) and (\ref{eq: hat_W^N0 dynamics - 2}),
\begin{align}
\dot{E}^{N_0}(t) & = ( A_0 - L C_0 ) E^{N_0}(t) - L C_1 \tilde{E}^{N-N_0}(t) \nonumber \\
& \phantom{=}\; - \alpha_1 L \begin{bmatrix} 1 & 0 & 0 \end{bmatrix} \hat{W}^{N_0}_{a} - L \zeta(t) . \label{eq: E^N0 dynamics - 2}
\end{align}
We now define $\hat{W}^{N-N_0}(t) = \begin{bmatrix} \hat{w}_{N_0 + 1}(t) & \ldots & \hat{w}_{N}(t) \end{bmatrix}^\top$, $A_2 = \mathrm{diag}(- \lambda_{N_0 + 1} + q_c , \ldots , - \lambda_{N} + q_c)$, $B_{2,a} = \begin{bmatrix} a_{N_0 + 1} & \ldots & a_{N} \end{bmatrix}^\top$, $B_{2,b} = \begin{bmatrix} b_{N_0 + 1} & \ldots & b_{N} \end{bmatrix}^\top$. We obtain from (\ref{eq: dirichlet boundary measurement - observer dynamics - 1}) with $l_n = 0$ for $N_0 +1 \leq n \leq N$ that
\begin{align}
& \dot{\hat{W}}^{N-N_0}(t)
= A_2 \hat{W}^{N-N_0}(t) + B_{2,a} u(t) + B_{2,b} v(t) \nonumber \\
& = A_2 \hat{W}^{N-N_0}(t) + \left( B_{2,b} K + \begin{bmatrix} B_{2,a} & 0 & 0 \end{bmatrix} \right) \hat{W}_a^{N_0}(t) \label{eq: hat_W^N-N0 dynamics}
\end{align}
and, using (\ref{eq: homogeneous RD system - spectral reduction - 2}) and (\ref{eq: dirichlet boundary measurement - observer dynamics - 1}),
\begin{equation}\label{eq: E^N-N0 dynamics}
\dot{\tilde{E}}^{N-N_0}(t) = A_2 \tilde{E}^{N-N_0}(t) .
\end{equation}
Putting now together (\ref{eq: hat_W_a^N0 dynamics - 2}-\ref{eq: E^N-N0 dynamics}) while introducing
\begin{equation}\label{eq: def X}
X(t) = \mathrm{col} \left( \hat{W}_a^{N_0}(t) , E^{N_0}(t) , \hat{W}^{N-N_0}(t) , \tilde{E}^{N-N_0}(t) \right) ,
\end{equation}
we obtain that
\begin{equation}\label{eq: dynamics closed-loop system - finite dimensional part}
\dot{X}(t) = F X(t) + \mathcal{L} \zeta(t) - \mathcal{L}_r r(t)
\end{equation}
where
\begin{equation*}
F = \begin{bmatrix}
A_\mathrm{cl}(\alpha_1) & \tilde{L} C_0 & 0 & \tilde{L} C_1 \\
-\alpha_1 L \begin{bmatrix} 1 & 0 & 0 \end{bmatrix} & A_0 - L C_0 & 0 & -L C_1 \\
B_{2,b} K + \begin{bmatrix} B_{2,a} & 0 & 0 \end{bmatrix} & 0 & A_2 & 0 \\
0 & 0 & 0 & A_2
\end{bmatrix} ,
\end{equation*}
\begin{equation*}
\mathcal{L} = \mathrm{col} ( \tilde{L} , - L , 0 , 0 ) , \quad
\mathcal{L}_r = \mathrm{col} ( B_r , 0 , 0 , 0 ) .
\end{equation*}

Defining $E = \begin{bmatrix} 1 & 0 & \ldots & 0\end{bmatrix}$ and $\tilde{K} = \begin{bmatrix} K & 0 & 0 & 0\end{bmatrix}$, we obtain from (\ref{eq: def hat W^N_0_a(t)}), (\ref{eq: v - state feedback}), and (\ref{eq: def X}) that
\begin{equation}\label{eq: u and v in function of X}
u(t) = E X(t) , \quad v(t) = \tilde{K} X(t)
\end{equation}
and we can introduce
\begin{equation}\label{eq: matrix G}
G = \Vert a \Vert_{L^2}^2 E^\top E + \Vert b \Vert_{L^2}^2 \tilde{K}^\top \tilde{K}  \preceq g I
\end{equation}
with $g = \Vert a \Vert_{L^2}^2 + \Vert b \Vert_{L^2}^2 \Vert K \Vert^2$ a constant independent of $N$.

\begin{lemma}\label{eq: Dirichlet measurement - Kalman condion}
$(A_1,B_1)$ is controllable and $(A_0,C_0)$ is observable.
\end{lemma}

\begin{proof}
From~\cite[Lem.~2]{lhachemi2020pi}, $(A_1,B_1)$ is controllable if and only if
\begin{equation}\label{eq: kalman condition previous work}
\left(\begin{bmatrix} 0 & 0 \\ B_{0,a} & A_0 \end{bmatrix} , \begin{bmatrix} 1 \\ B_{0,b} \end{bmatrix} \right)
\end{equation}
satisfies the Kalman condition and the matrix
$
T = \begin{bmatrix} 0 & 0 & 1 \\ B_{0,a} & A_0 & B_{0,b} \\ \alpha_0 & C_0 & \beta_0 \end{bmatrix}
$
is invertible. The former condition was assessed in~\cite{lhachemi2020finite}. Hence we focus on the latter one. Let $\begin{bmatrix} u_e & w_{1,e} & \ldots & w_{N_0,e} & v_e \end{bmatrix}^\top \in\mathrm{ker}(T)$. We obtain that
\begin{subequations}\label{eq: check Kalman condition}
\begin{align}
v_e & = 0 , \label{eq: check Kalman condition - 1} \\
a_n u_e + (-\lambda_n +q_c) w_{n,e} & = 0 , \quad 1 \leq n \leq N_0 , \label{eq: check Kalman condition - 2} \\
\alpha_0 u_e + \sum_{n=1}^{N_0} \phi_n(0) w_{n,e} & = 0 . \label{eq: check Kalman condition - 3}
\end{align}
\end{subequations}
Defining for $n \geq N_0 +1$ the quantity $w_{n,e} = -\frac{a_n}{-\lambda_n+q_c}u_e$, we have $(-\lambda_n+q_c)w_{n,e} + a_n u_e = 0$ for all $n \geq 1$. Hence $(w_{n,e})_{n \geq 1} , (\lambda_n w_{n,e})_{n \geq 1} \in l^2(\N)$ ensuring that $w_e \triangleq \sum_{n \geq 1} w_{n,e} \phi_n \in D(\mathcal{A})$ and $\mathcal{A} w_e = \sum_{n \geq 1} \lambda_n w_{n,e} \phi_n$. This shows that $- \mathcal{A} w_e + q_c w_e + a u_e = 0$. Moreover, from (\ref{eq: check Kalman condition - 3}) and using (\ref{eq: def alpha_0 and beta_0}), we infer that $w_e(0) = 0$. From the two last identities, we have that $(p w_e')' + (q_c - q) w_e + a u_e = 0$, $w_e(0) = w_e'(0) = 0$, and $w_e(1) = 0$. Introducing the change of variable $z_e(x) = w_e(x) + x^2 u_e$, we deduce that $(p z_e')' + (q_c - q) z_e = 0$, $z_e(0) = z_e'(0) = 0$, and $z_e(1) = u_e$. By Cauchy uniqueness, we infer that $z_e = 0$ hence $u_e = z_e(1) = 0$. Thus we have $w_e = z_e - x^2 u_e = 0$ hence $w_{n,e} = 0$ for all $n \geq 1$. We deduce that $\mathrm{ker}(T)=\{0\}$. Overall, we have shown that $(A_1,B_1)$ is controllable. Finally, the pair $(A_0,C_0)$ is observable because 1) $A_0$ is diagonal with simple eigenvalues, 2) by Cauchy uniqueness, $\phi_n(0) \neq 0$ for all $n \geq 1$.
\end{proof}

In the sequel we select, once for all and independely of the dimension $N$ of the observer, the gains $K \in\R^{1 \times (N_0 +2)}$ and $L \in\R^{N_0}$ so that $A_1 + B_1 K$ and $A_0 - L C_0$ are Hurwitz with eigenvalues that have a real part strictly less than $- \delta < 0$.

\subsection{Equilibirum condition and dynamics of deviations}

We aim at characterizing the equilibrium condition of the closed-loop system composed of the reaction-diffusion system (\ref{eq: dirichlet boundary measurement - RD system}), the auxiliary command input dynamics (\ref{eq: homogeneous RD system - abstract form - auxiliary input v}), the integral action (\ref{eq: dirichlet boundary measurement - integral component}), the observer dynamics (\ref{eq: dirichlet boundary measurement - observer dynamics - 1}), and the state-feedback (\ref{eq: v - state feedback}). To do so let $r(t) = r_e \in\R$ be arbitrary. We must solve the system of equations:
\begin{subequations}\label{eq: dirichlet boundary measurement - equilibirum condition}
\begin{align}
0 & = (-\lambda_n+q_c) w_{n,e} + a_n u_e + b_n v_e = 0 , \quad n \geq 1 , \label{eq: dirichlet boundary measurement - equilibirum condition - 1} \\
0 & = v_e = K \hat{W}^{N_0}_{a,e} , \label{eq: dirichlet boundary measurement - equilibirum condition - 2} \\
0 & = \sum_{n=1}^{N_0} \phi_n(0) \hat{w}_{n,e} + \alpha_0 u_e + \beta_0 v_e - r_e , \label{eq: dirichlet boundary measurement - equilibirum condition - 3} \\
0 & = (-\lambda_n+q_c) \hat{w}_{n,e} + a_n u_e + b_n v_e \nonumber \\
& \phantom{=}\; - l_n \left\{ \sum_{i=1}^{N} \phi_i(0) \hat{w}_{i,e} - \alpha_1 u_e - \tilde{y}_e \right\} , \quad 1 \leq n \leq N_0 , \label{eq: dirichlet boundary measurement - equilibirum condition - 4} \\
0 & = (-\lambda_n+q_c) \hat{w}_{n,e} + a_n u_e + b_n v_e , \quad N_0 + 1 \leq n \leq N , \label{eq: dirichlet boundary measurement - equilibirum condition - 5} \\
\tilde{y}_e & = \sum_{n \geq 1} \phi_n(0) w_{n,e} . \label{eq: dirichlet boundary measurement - equilibirum condition - 6}
\end{align}
\end{subequations}
We first note from (\ref{eq: dirichlet boundary measurement - equilibirum condition - 2}) that $v_e = 0$. Then, from (\ref{eq: dirichlet boundary measurement - equilibirum condition - 1}) we have $w_{n,e} = -\frac{a_n}{-\lambda_n +q_c}u_e$ for all $n \geq N_0 +1$. In particular, from (\ref{eq: dirichlet boundary measurement - equilibirum condition - 5}), we have $\hat{w}_{n,e} = w_{n,e} = -\frac{a_n}{-\lambda_n +q_c}u_e$ for all $N_0 + 1 \leq n \leq N$. Defining $e_{n,e} = w_{n,e} - \hat{w}_{n,e}$ and $\zeta_e = \sum_{n \geq N+1} \phi_n(0) w_{n,e}$, we obtain that $e_{n,e} = 0$ for all $N_0 + 1 \leq n \leq N$. Hence, from (\ref{eq: dirichlet boundary measurement - equilibirum condition - 4}), we infer that $0 = (-\lambda_n + q_c) \hat{w}_{n,e} + a_n u_e + l_n \sum_{i = 1}^{N_0} \phi_i(0) e_{i,e} + l_n \alpha_1 u_e + l_n \zeta_e$ for all $1 \leq n \leq N_0$. Combining this latter identity with (\ref{eq: dirichlet boundary measurement - equilibirum condition - 1}), we obtain that $(A_0-LC_0)E^{N_0}_e - L \alpha_1 u_e - L \zeta_e = 0$. Invoking (\ref{eq: dirichlet boundary measurement - def alpha_1}), we note that $\alpha_1 u_e = - \sum_{n \geq N+1} \phi_n(0) w_{n,e} = - \zeta_e$, implying that $(A_0-LC_0)E^{N_0}_e = 0$. Since $A_0-LC_0$ is Hurwitz, we infer that $e_{n,e} = 0$ for all $1 \leq n \leq N_0$. In particular, $\hat{w}_{n,e} = w_{n,e}$ for all $1 \leq n \leq N$. From (\ref{eq: dirichlet boundary measurement - equilibirum condition - 2}-\ref{eq: dirichlet boundary measurement - equilibirum condition - 4}) we deduce that $0 = A_\mathrm{cl}(\alpha_1) \hat{W}_{a,e}^{N_0} - B_r r_e + \tilde{L}\zeta_e$. Recalling that $\zeta_e = - \alpha_1 u_e$ and $A_\mathrm{cl}(\alpha_1)$ is defined by (\ref{eq: def Acl}), we obtain that $(A_1+B_1 K) \hat{W}_{a,e}^{N_0} = B_r r_e$. Since $A_1+B_1 K$ is Hurwitz, we infer that $\hat{W}_{a,e}^{N_0} = \begin{bmatrix} u_e & \hat{w}_{1,e} & \ldots & \hat{w}_{N_0,e} & \xi_e \end{bmatrix}^\top = (A_1+B_1 K)^{-1} B_r r_e$. This is in particular compatible with (\ref{eq: dirichlet boundary measurement - equilibirum condition - 2}) since, based on (\ref{eq: def matrices A1 B1 and Br}), we indeed obtain that $K \hat{W}^{N_0}_{a,e} = 0$. We note that $(w_{n,e})_{n \geq 1} , (\lambda_n w_{n,e})_{n \geq 1} \in l^2(\N)$ ensuring that $w_e \triangleq \sum_{n \geq 1} w_{n,e} \phi_n \in D(\mathcal{A})$ and $\mathcal{A} w_e = \sum_{n \geq 1} \lambda_n w_{n,e} \phi_n$. Using (\ref{eq: dirichlet boundary measurement - equilibirum condition - 1}), we obtain that $- \mathcal{A} w_e + q_c w_e + a u_e + b v_e = 0$. Introducing the change of variable $z_e = w_e + x^2 u_e$, $z_e$ is a static solution of (\ref{eq: dirichlet boundary measurement - RD system - 1}-\ref{eq: dirichlet boundary measurement - RD system - 2}) associated with the constant control input $u(t) = u_e$. Denoting by $y_e \triangleq z_e(0) = w_e(0) = \tilde{y}_e$, we infer from (\ref{eq: dirichlet boundary measurement - equilibirum condition - 3}) while invoking (\ref{eq: def alpha_0 and beta_0}) that
$r_e = \sum_{n=1}^{N_0} \phi_n(0) \hat{w}_{n,e} + \alpha_0 u_e
= \sum_{n \geq 1} \phi_n(0) w_{n,e}
= y_e$.
Hence, for an arbitrarily given constant reference signal $r(t) = r_e \in\R$, the equilibirum condition of the closed-loop system is unique, fully characterized by $r_e$, and is such that $y_e = r_e$.

We can now introduce the dynamics of deviation of the different quantities w.r.t. the equilibrium condition characterized by $r_e \in\R$. In particular:
\begin{subequations}\label{eq: dirichlet boundary measurement - dynamics of deviations}
\begin{align}
& \Delta w(t,x) = \Delta z(t,x) - x^2 \Delta u(t) , \label{eq: dirichlet boundary measurement - dynamics of deviations - 1} \\
& \Delta \dot{X}(t) = F \Delta X(t) + \mathcal{L} \Delta \zeta(t) - \mathcal{L}_r \Delta r(t) , \label{eq: dirichlet boundary measurement - dynamics of deviations - 2} \\
& \Delta\zeta(t) = \sum_{n \geq N+1} \phi_n(0) \Delta w_n(t) , \label{eq: dirichlet boundary measurement - dynamics of deviations - 3} \\
& \Delta \dot{w}_n(t) = (-\lambda_n + q_c) \Delta w_n(t) + a_n \Delta u(t) + b_n \Delta v(t) , \label{eq: dirichlet boundary measurement - dynamics of deviations - 4} \\
& \Delta v(t) = K \Delta \hat{W}_a^{N_0}(t) , \label{eq: dirichlet boundary measurement - dynamics of deviations - 5} \\
& \Delta\tilde{y}(t) = \Delta y(t) = \sum_{n \geq 1} \phi_n(0) \Delta w_n(t)  \label{eq: dirichlet boundary measurement - dynamics of deviations - 6}
\end{align}
\end{subequations}
with $\Delta w_n(t) = \left< \Delta w(t,\cdot) , e_n \right>$.

\subsection{Stability analysis and regulation assessment}

We define the constant $M_{1,\phi} = \sum_{n \geq N+1} \frac{\phi_n(0)^2}{\lambda_n}$, which is finite when $p \in \mathcal{C}^2([0,1])$ because $\phi_n(0) = O(1)$ as $n \rightarrow + \infty$ and (\ref{eq: estimation lambda_n}) hold.

\begin{theorem}\label{thm: Case of a Dirichlet boundary measurement - stab}
Let $p \in \mathcal{C}^2([0,1])$ with $p > 0$, $q \in \mathcal{C}^0([0,1])$ with $q \geq 0$, and $q_c \in \R$. Consider the reaction-diffusion system described by (\ref{eq: dirichlet boundary measurement - RD system}). Let $N_0 \geq 1$ and $\delta > 0$ be given such that $- \lambda_n + q_c < -\delta < 0$ for all $n \geq N_0 +1$. Let $K \in\R^{1 \times (N_0 +2)}$ and $L \in\R^{N_0}$ be such that $A_1 + B_1 K$ and $A_0 - L C_0$ are Hurwitz with eigenvalues that have a real part strictly less than $-\delta < 0$. For a given $N \geq N_0 +1$, assume that there exist $P \succ 0$, $\alpha>1$, and $\beta,\gamma > 0$ such that
\begin{subequations}\label{eq: thm1 - Theta}
\begin{align}
\Theta_{1} & = \begin{bmatrix} F^\top P + P F + 2 \delta P + \alpha \gamma G & P \mathcal{L} \\ \mathcal{L}^\top P^\top & -\beta \end{bmatrix} \prec 0 , \label{eq: thm1 - Theta_1}  \\
\Theta_2 & = 2\gamma \left\{ - \left( 1 - \frac{1}{\alpha} \right) \lambda_{N+1} + q_c + \delta \right\} + \beta M_{1,\phi} \leq 0 . \label{eq: thm1 - Theta_2}
\end{align}
\end{subequations}
Then, for any $\eta \in [0,1)$, there exists $M > 0$ such that, for any $z_0 \in H^2(0,1)$ and $u(0),\xi(0),\hat{w}_n(0)\in\R$ such that $z_0'(0)=0$ and $z_0(1) = u(0)$, the classical solution of the closed-loop system composed of the plant (\ref{eq: dirichlet boundary measurement - RD system}), the integral actions (\ref{eq: homogeneous RD system - abstract form - auxiliary input v}) and (\ref{eq: dirichlet boundary measurement - integral component}), the observer dynamics (\ref{eq: dirichlet boundary measurement - observer dynamics - 1}), and the state feedback (\ref{eq: v - state feedback}) satisfies
\begin{align}
& \Delta u(t)^2 + \Delta  \xi(t)^2 + \sum_{n=1}^{N} \Delta \hat{w}_n(t)^2 + \Vert \Delta z(t) \Vert_{H^1}^2 \nonumber \\
& \quad\leq M e^{-2 \delta t} \left( \Delta u(0)^2 + \Delta \xi(0)^2 + \sum_{n=1}^{N} \Delta \hat{w}_n(0)^2 + \Vert \Delta z_0 \Vert_{H^1}^2 \right) \nonumber \\
& \phantom{\quad\leq}\; + M \sup_{\tau\in[0,t]} e^{-2\eta\delta(t-\tau)} \Delta r(\tau) ^2 \label{eq: dirichlet boundary measurement - stab result}
\end{align}
for all $t \geq 0$. Moreover, the above constraints are always feasible for $N$ large enough.
\end{theorem}

\begin{proof}
Let $P \succ 0$ and $\gamma > 0$ and consider the Lyapunov function candidate defined by
\begin{equation}\label{eq: Lyap function for H1 stab}
V(\Delta X,\Delta w) = \Delta X^\top P \Delta X + \gamma \sum_{n \geq N+1} \lambda_n \left< \Delta w , \phi_n \right>^2 .
\end{equation}
with $\Delta X\in\R^{2N+2}$ and $\Delta w \in D(\mathcal{A})$. The first term accounts for the dynamics of the truncated model (\ref{eq: dynamics closed-loop system - finite dimensional part}) while the series, which in view of (\ref{eq: inner product Af and f}) is related to the $H^1$ norm of the PDE trajectories, is used to handle the modes $w_n$ for $n \geq N+1$ of the PDE plant. Proceeding exactly as in~\cite{lhachemi2020finite} but taking into account the extra contribution of the reference signal appearing in (\ref{eq: dynamics closed-loop system - finite dimensional part}), we obtain for $t \geq 0$ that
\begin{align}
& \dot{V}(t) + 2 \delta V(t)
\leq \begin{bmatrix} \Delta X(t) \\ \Delta \zeta(t) \end{bmatrix}^\top \Theta_{1} \begin{bmatrix} \Delta X(t) \\ \Delta \zeta(t) \end{bmatrix} \nonumber \\
& - 2 \Delta X(t)^\top P \mathcal{L}_r \Delta r(t)
+ \sum_{n \geq N+1} \lambda_n \Gamma_n \Delta w_n(t)^2 \label{eq: Lyap time derivative}
\end{align}
with $\Gamma_n = 2\gamma \left\{ - \left( 1 - \frac{1}{\alpha} \right) \lambda_n + q_c + \delta \right\} + \beta M_{1,\phi}$ for $n \geq N+1$, $\alpha>1$ and $\beta>0$ arbitrary, and where, with a slight abuse of notation, $\dot{V}(t)$ denotes the time derivative of $V(X(t),w(t))$ along the system trajectories (\ref{eq: dirichlet boundary measurement - dynamics of deviations}). Since $\alpha > 1$ we have $\Gamma_n \leq \Theta_2 \leq 0$ for all $n \geq N+1$. From (\ref{eq: thm1 - Theta_1}), there exists $\epsilon > 0 $ such that $\Theta_{1} \preceq - \epsilon I$. Hence the assumptions imply that $\dot{V}(t) + 2 \delta V(t) \leq - \epsilon \Vert \Delta X(t) \Vert^2 - 2 \Delta X(t)^\top P \mathcal{L}_r \Delta r(t) \leq \frac{\Vert P\mathcal{L}_r \Vert^2}{\epsilon} \Delta r(t)^2$ where Young's inequality has been used to derive the latter estimate. After integration, we obtain for any $\eta \in [0,1)$ the existence of a constant $M_1 > 0$ such that $V(t) \leq e^{-2\delta t} V(0) + M_1 \sup_{\tau\in[0,t]} e^{-2\eta\delta(t-\tau)} \Delta r(\tau)^2$ for all $t \geq 0$. The claimed estimate (\ref{eq: dirichlet boundary measurement - stab result}) easily follows from the definition (\ref{eq: Lyap function for H1 stab}) of the Lyapunov function, the use of (\ref{eq: inner product Af and f}), Poincar{\'e}'s inequality, and the change of variable (\ref{eq: dirichlet boundary measurement - dynamics of deviations - 1}).

We now show that we can always select $N \geq N_0 + 1$, $P \succ 0$, $\alpha > 1$, and $\beta,\gamma > 0$ such that (\ref{eq: thm1 - Theta}) holds. By the Schur complement, $\Theta_{1} \prec 0$ is equivalent to  $F^\top P + P F + 2 \delta P + \alpha\gamma G + \frac{1}{\beta} P \mathcal{L} \mathcal{L}^\top P^\top \prec 0$. We define $F = F_1 + F_2$ where
\begin{subequations}\label{eq: matrices F1 and F2}
\begin{equation}
F_1 = \begin{bmatrix}
A_1 + B_1 K & \tilde{L} C_0 & 0 & \tilde{L} C_1 \\
0 & A_0 - L C_0 & 0 & -L C_1 \\
B_{2,b} K + \begin{bmatrix} B_{2,a} & 0 & 0 \end{bmatrix} & 0 & A_2 & 0 \\
0 & 0 & 0 & A_2
\end{bmatrix} ,
\end{equation}
\begin{equation}
F_2 = \begin{bmatrix}
\alpha_1 \tilde{L} \begin{bmatrix} 1 & 0 & 0 \end{bmatrix} & 0 & 0 & 0 \\
-\alpha_1 L \begin{bmatrix} 1 & 0 & 0 \end{bmatrix} & 0 & 0 & 0 \\
0 & 0 & 0 & 0 \\
0 & 0 & 0 & 0
\end{bmatrix}
\end{equation}
\end{subequations}
with $\Vert F_2 \Vert \rightarrow 0$, because $\alpha_1 \rightarrow 0$, when $N \rightarrow + \infty$. We note that $A_1 + B_1 K + \delta I$ and $A_0 - L C_0 + \delta I$ are Hurwitz while $\Vert e^{(A_2+\delta I) t} \Vert \leq e^{-\kappa_0 t}$ with $\kappa_0 = \lambda_{N_0+1} - q_c -\delta > 0$. Moreover, $\Vert \tilde{L} C_1 \Vert \leq \Vert L \Vert \Vert C_1 \Vert$, $\Vert L C_1 \Vert \leq  \Vert L \Vert \Vert C_1 \Vert$, with $\Vert C_1 \Vert = O(1)$ as $N \rightarrow + \infty$ while $\Vert B_{2,b} K + \begin{bmatrix} B_{2,a} & 0 & 0 \end{bmatrix} \Vert \leq \Vert b \Vert_{L^2} \Vert K \Vert + \Vert a \Vert_{L^2}$ where the right-hand side is a constant independent of $N$. The application of~\cite[Lemma in Appendix]{lhachemi2020finite}, which is a generalization of a result found in~\cite{katz2020constructive}, to the matrix $F_1 + \delta I$ gives the existence of $P \succ 0$ such that $F_1^\top P + P F_1 + 2 \delta P = -I$ and $\Vert P \Vert = O(1)$ as $N \rightarrow + \infty$. Therefore, we have $F^\top P + P F + 2 \delta P + \alpha\gamma G + \frac{1}{\beta} P \mathcal{L} \mathcal{L}^\top P^\top = -I + F_2^\top P + P F_2 + \alpha\gamma G + \frac{1}{\beta} P \mathcal{L} \mathcal{L}^\top P^\top$ where $G$ satisfies (\ref{eq: matrix G}) and $\Vert \mathcal{L} \Vert = \sqrt{2} \Vert L \Vert$, which is independent of $N$. Hence, fixing arbitrarily $\alpha > 1$ while setting $\beta = \sqrt{N}$ and $\gamma = N^{-1}$, we infer that (\ref{eq: thm1 - Theta}) holds for $N \geq N_0 + 1$ large enough.
\end{proof}

\begin{remark}
Constraints (\ref{eq: thm1 - Theta}) are nonlinear w.r.t. the decision variables. However, for a given value of $N \geq N_0 + 1$ and arbitrarily fixing $\alpha > 1$, constraints (\ref{eq: thm1 - Theta}) become LMI conditions w.r.t. the decision variables $P \succ 0$ and $\beta,\gamma > 0$. As shown in the proof of Theorem~\ref{thm: Case of a Dirichlet boundary measurement - stab}, these latter LMI conditions are always feasible provided the order of the observer $N$ is selected large enough. Similar remarks apply to Theorems~\ref{thm: Case of a Neumann boundary measurement - stab} and~\ref{thm: crossed}.
\end{remark}

We now assess the setpoint regulation of the left Dirichlet trace.

\begin{theorem}\label{thm: Case of a Dirichlet boundary measurement - reg}
Under both assumptions and conclusions of Theorem~\ref{thm: Case of a Dirichlet boundary measurement - stab}, for any $\eta \in [0,1)$, there exists $M_r > 0$ such that
\begin{align}
& \vert y(t) - r(t) \vert \leq M_r e^{-\delta t} \bigg( \vert \Delta u(0) \vert + \vert \Delta \xi(0) \vert + \sum_{n=1}^{N} \vert \Delta \hat{w}_n(0) \vert \nonumber \\
& \qquad\qquad + \Vert \Delta z_0 \Vert_{H^1} \bigg) + M_r \sup_{\tau\in[0,t]} e^{-\eta\delta(t-\tau)} \vert \Delta r(\tau) \vert \label{eq: dirichlet boundary measurement - reg result}
\end{align}
for all $t \geq 0$.
\end{theorem}

\begin{proof}
Recalling that $y_e = r_e$, one has $\vert y(t) - r(t) \vert \leq \vert \Delta y(t) \vert + \vert \Delta r(t) \vert$. From (\ref{eq: dirichlet boundary measurement - dynamics of deviations - 6}) and Cauchy-Schwarz inequality, we infer that $\vert \Delta y(t) \vert \leq \sqrt{\sum_{n \geq 1} \frac{\phi_n(0)^2}{\lambda_n}} \sqrt{\sum_{n \geq 1} \lambda_n \Delta w_n(t)^2}$. Using now (\ref{eq: inner product Af and f}) we infer the existence of a constant $M_2 > 0$ such that $\vert \Delta y(t) \vert \leq M_2 \Vert \Delta w(t) \Vert_{H^1}$. The proof is completed by invoking the change of variable (\ref{eq: dirichlet boundary measurement - dynamics of deviations - 1}) and the stability result (\ref{eq: dirichlet boundary measurement - stab result}).
\end{proof}

\section{Neumann measurement and regulation control}\label{sec: Neuman}
We now consider the reaction-diffusion system with Neumann boundary observation described for $t > 0$ and $x \in (0,1)$ by
\begin{subequations}\label{eq: neumann boundary measurement - RD system}
\begin{align}
z_t(t,x) & = \left( p(x) z_x(t,x) \right)_x + (q_c - q(x)) z(t,x) \label{eq: neumann boundary measurement - RD system - 1} \\
z(t,0) & = 0 , \quad z(t,1) = u(t) \label{eq: neumann boundary measurement - RD system - 2} \\
z(0,x) & = z_0(x) \\
y(t) & = z_x(t,0)
\end{align}
\end{subequations}
in the case $p \in \mathcal{C}^2([0,1])$.

\subsection{Control design}
Introducing the change of variable
\begin{equation}\label{eq: neumann boundary measurement - change of variable}
w(t,x) = z(t,x) - x u(t)
\end{equation}
we obtain
\begin{subequations}\label{eq: neumann boundary measurement - homogeneous RD system}
\begin{align}
& w_t(t,x) = ( p(x) w_x(t,x) )_x + (q_c - q(x)) w(t,x) \nonumber \\
& \phantom{w_t(t,x) =}\; + a(x) u(t) + b(x) \dot{u}(t) \label{eq: neumann boundary measurement - homogeneous RD system - 1} \\
& w(t,0) = 0 , \quad w(t,1) = 0 \label{eq: neumann boundary measurement - homogeneous RD system - 2} \\
& w(0,x) = w_0(x) \label{eq: neumann boundary measurement - homogeneous RD system - 3} \\
& \tilde{y}(t) = w_x(t,0) \label{eq: neumann boundary measurement - homogeneous RD system - 4}
\end{align}
\end{subequations}
with $a,b \in L^2(0,1)$ defined by $a(x) = p'(x) + (q_c-q(x))x$ and $b(x) = -x$, respectively, $\tilde{y}(t) = y(t) - u(t)$ , and $w_0(x) = z_0(x) - x u(0)$. Introducing the auxiliary command input $v(t) = \dot{u}(t)$, we infer that (\ref{eq: homogeneous RD system - abstract form}) still holds but the domain of $\mathcal{A}$ is now replaced by $D(\mathcal{A}) = \{ f \in H^2(0,1) \,:\, f(0)=f(1)=0 \}$. Then, considering classical solutions associated with any $z_0 \in H^2(0,1)$ and any $u(0) \in \R$ such that $z_0(0)=0$ and $z_0(1) = u(0)$ (their existence for the upcoming closed-loop dynamics is an immediate consequence of \cite[Chap.~6, Thm.~1.7]{pazy2012semigroups}), (\ref{eq: homogeneous RD system - spectral reduction - 1}-\ref{eq: homogeneous RD system - spectral reduction - 2}) is still valid while (\ref{eq: homogeneous RD system - spectral reduction - 3}) is replaced by
\begin{equation}\label{eq: neumann boundary measurement - measurement bis}
\tilde{y}(t) = \sum_{i \geq 1} \phi_i'(0) w_i(t) .
\end{equation}
Based on similar motivations than the ones reported in Section~\ref{sec: Dirichlet measurement and regulation control}, we consider the integral component
\begin{align}\label{eq: neumann boundary measurement - integral component}
\dot{\xi}(t)
& = \sum_{n = 1}^{N_0} \phi_n'(0) \hat{w}_n(t) + \alpha_0 u(t) + \beta_0 v(t) - r(t) .
\end{align}
with
\begin{equation}\label{eq: neumann boundary measurement - def alpha_0 and beta_0}
\alpha_0 = 1 - \sum_{n \geq N_0 + 1} \frac{a_n \phi_n'(0)}{-\lambda_n + q_c} , 
\quad \beta_0 = - \sum_{n \geq N_0 + 1} \frac{b_n \phi_n'(0)}{-\lambda_n + q_c}
\end{equation}
and where the observation dynamics, for $1 \leq n \leq N$, take the form:
\begin{align}
\dot{\hat{w}}_n (t)
& = ( -\lambda_n + q_c ) \hat{w}_n(t) + a_n u(t) + b_n v(t) \nonumber \\
& \phantom{=}\; - l_n \left( \sum_{i=1}^N \phi_i'(0) \hat{w}_i(t) - \alpha_1 u(t) - \tilde{y}(t) \right) \label{eq: neumann boundary measurement - observer dynamics - 1}
\end{align}
with
\begin{equation}\label{eq: neumann boundary measurement - def alpha_1}
\alpha_1 = \sum_{n \geq N +1} \dfrac{a_n \phi_n'(0)}{-\lambda_n + q_c}
\end{equation}
and where $l_n \in\R$ are the observer gains. We set $l_n = 0$ for $N_0+1 \leq n \leq N$. Proceeding now as in Section~\ref{sec: Dirichlet measurement and regulation control} but with the updated versions of the matrices $C_0$ and $C_1$ now given by $C_0 = \begin{bmatrix} \phi_1'(0) & \ldots & \phi_{N_0}'(0) \end{bmatrix}$ and $C_1 = \begin{bmatrix} \dfrac{\phi_{N_0 +1}'(0)}{\lambda_{N_0 + 1}} & \ldots & \dfrac{\phi_{N}'(0)}{\lambda_{N}} \end{bmatrix}$ while redefining $\tilde{e}_n(t)$ and $\zeta(t)$ as $\tilde{e}_n(t) = \lambda_n e_n(t)$ and $\zeta(t) = \sum_{n \geq N+1} \phi_n'(0) w_n(t)$, we infer that (\ref{eq: dynamics closed-loop system - finite dimensional part}) holds.

\begin{lemma}
$(A_1,B_1)$ is controllable and $(A_0,C_0)$ is observable.
\end{lemma}

The proof of this Lemma is analogous to the one of Lemma~\ref{eq: Dirichlet measurement - Kalman condion} and is thus omitted. We select in the sequel $K \in\R^{1 \times (N_0 +2)}$ and $L \in\R^{N_0}$ such that $A_1 + B_1 K$ and $A_0 - L C_0$ are Hurwitz.

\subsection{Equilibirum condition and dynamics of deviations}

Proceeding similarly to Section~\ref{sec: Dirichlet measurement and regulation control}, we can characterize the equilibrium condition of the closed-loop system composed of the reaction-diffusion system (\ref{eq: neumann boundary measurement - RD system}), the auxiliary command input dynamics (\ref{eq: homogeneous RD system - abstract form - auxiliary input v}), the integral action (\ref{eq: neumann boundary measurement - integral component}), the observer dynamics (\ref{eq: neumann boundary measurement - observer dynamics - 1}), and the state-feedback (\ref{eq: v - state feedback}). In particular, setting $r(t) = r_e \in\R$, it can be shown that there exists a unique solution to:
\begin{subequations}\label{eq: neumann boundary measurement - equilibirum condition}
\begin{align}
0 & = (-\lambda_n+q_c) w_{n,e} + a_n u_e + b_n v_e = 0 , \quad n \geq 1 , \label{eq: neumann boundary measurement - equilibirum condition - 1} \\
0 & = v_e = K \hat{W}^{N_0}_{a,e} , \label{eq: neumann boundary measurement - equilibirum condition - 2} \\
0 & = \sum_{n=1}^{N_0} \phi_n'(0) \hat{w}_{n,e} + \alpha_0 u_e + \beta_0 v_e - r_e , \label{eq: neumann boundary measurement - equilibirum condition - 3} \\
0 & = (-\lambda_n+q_c) \hat{w}_{n,e} + a_n u_e + b_n v_e \nonumber \\
& \phantom{=}\; - l_n \left\{ \sum_{i=1}^{N} \phi_i'(0) \hat{w}_{i,e} - \alpha_1 u_e - \tilde{y}_e \right\} , \quad 1 \leq n \leq N_0 , \label{eq: neumann boundary measurement - equilibirum condition - 4} \\
0 & = (-\lambda_n+q_c) \hat{w}_{n,e} + a_n u_e + b_n v_e , \quad N_0 + 1 \leq n \leq N , \label{eq: neumann boundary measurement - equilibirum condition - 5} \\
\tilde{y}_e & = \sum_{n \geq 1} \phi_n'(0) w_{n,e} . \label{eq: neumann boundary measurement - equilibirum condition - 6}
\end{align}
\end{subequations}
Moreover we can define $w_e \triangleq \sum_{n \geq 1} w_{n,e} \phi_n \in D(\mathcal{A})$. Introducing the change of variable $z_e = w_e + x u_e$, $z_e$ is a static solution of (\ref{eq: neumann boundary measurement - RD system - 1}-\ref{eq: neumann boundary measurement - RD system - 2}) associated with the constant control input $u(t) = u_e$. Denoting by $y_e \triangleq z_e'(0)$, we also infer that $y_e = r_e$, achieving the desired reference tracking. This allows the introduction of the dynamics of deviation of the different quantities w.r.t. the equilibrium condition characterized by $r_e \in\R$. We have:
\begin{subequations}\label{eq: neumann boundary measurement - dynamics of deviations}
\begin{align}
& \Delta w(t,x) = \Delta z(t,x) - x \Delta u(t) , \label{eq: neumann boundary measurement - dynamics of deviations - 1} \\
& \Delta \dot{X}(t) = F \Delta X(t) + \mathcal{L} \Delta \zeta(t) - \mathcal{L}_r \Delta r(t) , \label{eq: neumann boundary measurement - dynamics of deviations - 2} \\
& \Delta\zeta(t) = \sum_{n \geq N+1} \phi_n'(0) \Delta w_n(t) , \label{eq: neumann boundary measurement - dynamics of deviations - 3} \\
& \Delta \dot{w}_n(t) = (-\lambda_n + q_c) \Delta w_n(t) + a_n \Delta u(t) + b_n \Delta v(t) , \label{eq: neumann boundary measurement - dynamics of deviations - 4} \\
& \Delta v(t) = K \Delta \hat{W}_a^{N_0}(t) , \label{eq: neumann boundary measurement - dynamics of deviations - 5} \\
& \Delta\tilde{y}(t) = \Delta y(t) - \Delta u(t) = \sum_{n \geq 1} \phi_n'(0) \Delta w_n(t) . \label{eq: neumann boundary measurement - dynamics of deviations - 6}
\end{align}
\end{subequations}

\subsection{Stability analysis and regulation assessment}

We define, for any $\epsilon \in (0,1/2]$, the constant $M_{2,\phi}(\epsilon) = \sum_{n \geq N+1} \frac{\phi_n'(0)^2}{\lambda_n^{3/2+\epsilon}}$, which is finite when $p \in \mathcal{C}^2([0,1])$ because we recall that $\phi_n'(0) = O(\sqrt{\lambda_n})$ as $n \rightarrow + \infty$ and (\ref{eq: estimation lambda_n}) hold.

\begin{theorem}\label{thm: Case of a Neumann boundary measurement - stab}
Let $p \in \mathcal{C}^2([0,1])$ with $p > 0$, $q \in \mathcal{C}^0([0,1])$ with $q \geq 0$, and $q_c \in \R$. Consider the reaction-diffusion system described by (\ref{eq: neumann boundary measurement - RD system}). Let $N_0 \geq 1$ and $\delta > 0$ be given such that $- \lambda_n + q_c < -\delta < 0$ for all $n \geq N_0 +1$. Let $K \in\R^{1 \times (N_0 +2)}$ and $L \in\R^{N_0}$ be such that $A_1 + B_1 K$ and $A_0 - L C_0$ are Hurwitz with eigenvalues that have a real part strictly less than $-\delta < 0$. For a given $N \geq N_0 +1$, assume that there exist $P \succ 0$, $\epsilon \in (0,1/2]$, $\alpha > 1$, and $\beta,\gamma > 0$ such that $\Theta_{1} \prec 0$, where $\Theta_{1}$ is defined by (\ref{eq: thm1 - Theta_1}),
\begin{align*}
\Theta_2 & = 2\gamma \left\{ - \left( 1  - \frac{1}{\alpha} \right) \lambda_{N+1} + q_c + \delta \right\} + \beta M_{2,\phi}(\epsilon) \lambda_{N+1}^{1/2+\epsilon} \leq 0 , \\
\Theta_3 & = 2\gamma\left( 1  - \frac{1}{\alpha} \right) - \frac{\beta M_{2,\phi}(\epsilon)}{\lambda_{N+1}^{1/2-\epsilon}} \geq 0 .
\end{align*}
Then, for any $\eta \in [0,1)$, there exists $M > 0$ such that, for any $z_0 \in H^2(0,1)$ and $u(0),\xi(0),\hat{w}_n(0)\in\R$ such that $z_0(0)=0$ and $z_0(1) = u(0)$, the classical solution of the closed-loop system composed of the plant (\ref{eq: neumann boundary measurement - RD system}), the integral actions (\ref{eq: homogeneous RD system - abstract form - auxiliary input v}) and (\ref{eq: neumann boundary measurement - integral component}), the observer dynamics (\ref{eq: neumann boundary measurement - observer dynamics - 1}), and the state feedback (\ref{eq: v - state feedback}) satisfies  (\ref{eq: dirichlet boundary measurement - stab result}) for all $t \geq 0$. Moreover, the above constraints are always feasible for $N$ large enough.
\end{theorem}

\begin{proof}
Let $P \succ 0$ and $\gamma > 0$ and consider the Lyapunov function candidate defined by (\ref{eq: Lyap function for H1 stab}). Then, proceeding as in~\cite{lhachemi2020finite} but taking into account the extra contribution of the reference signal appearing in (\ref{eq: dynamics closed-loop system - finite dimensional part}), we obtain that (\ref{eq: Lyap time derivative}) holds for all $t \geq 0$ with $\Gamma_n = 2\gamma \left\{ - \left( 1  - \frac{1}{\alpha} \right) \lambda_n + q_c + \delta \right\} + \beta M_{2,\phi}(\epsilon) \lambda_n^{1/2+\epsilon}$. Since $\epsilon \in (0,1/2]$, we have $\lambda_n^{1/2+\epsilon} = \lambda_n/\lambda_n^{1/2-\epsilon} \leq \lambda_n/\lambda_{N+1}^{1/2-\epsilon}$ for all $n \geq N+1$, hence $\Gamma_n \leq - \Theta_3 \lambda_n + 2\gamma \{ q_c + \delta \} \leq \Theta_2 \leq 0$ for all $n \geq N+1$ where we have used that $\Theta_3 \geq 0$. Now the proof of the stability estimate (\ref{eq: dirichlet boundary measurement - stab result}) is analogous to the one reported in the proof of Theorem~\ref{thm: Case of a Dirichlet boundary measurement - stab}.

It remains to show that we can always select $N \geq N_0 + 1$, $P \succ 0$, $\epsilon \in (0,1/2]$, $\alpha > 1$, and $\beta,\gamma > 0$ such that $\Theta_1 \prec 0$, $\Theta_2 \leq 0$, and $\Theta_3 \geq 0$. To handle the constraint $\Theta_{1} \prec 0$, we proceed as in the last part of the proof of Theorem~\ref{thm: Case of a Dirichlet boundary measurement - stab}. This is allowed because $\Vert C_1 \Vert = O(1)$ as $N \rightarrow + \infty$. We set $\epsilon = 1/8$ and we arbitrary fix $\alpha > 1$. Setting $\beta = N^{1/8}$ and $\gamma = N^{-3/16}$, we deduce the existence of an integer $N \geq N_0 + 1$ large enough such that $\Theta_1 \prec 0$, $\Theta_2 \leq 0$, and $\Theta_3 \geq 0$.
\end{proof}

We are now in position to assess the setpoint regulation control of the left Dirichlet trace.

\begin{theorem}\label{thm: Case of a Neumann boundary measurement - reg}
Under both assumptions and conclusions of Theorem~\ref{thm: Case of a Neumann boundary measurement - stab}, for any $\eta \in [0,1)$, there exists $M_r > 0$ such that
\begin{align}
& \vert y(t) - r(t) \vert 
\leq M_r e^{-\delta t} \bigg( \vert \Delta u(0) \vert + \vert \Delta \xi(0) \vert + \sum_{n=1}^{N} \vert \Delta \hat{w}_n(0) \vert \nonumber \\
& + \Vert \Delta z_0 \Vert_{H^1} + \Vert \mathcal{A} \Delta w_0 \Vert_{L^2}  \bigg) + M_r \sup_{\tau\in[0,t]} e^{-\eta\delta(t-\tau)} \vert \Delta r(\tau) \vert \label{eq: neumann boundary measurement - reg result}
\end{align}
for all $t \geq 0$ where $\Delta w_0 = \Delta z_0 - x \Delta u(0)$.
\end{theorem}

\begin{proof}
Recalling that $y_e = r_e$, one has $\vert y(t) - r(t) \vert \leq \vert \Delta y(t) \vert + \vert \Delta r(t) \vert$. We infer from (\ref{eq: neumann boundary measurement - dynamics of deviations - 6}) and Cauchy-Schwarz inequality that $\vert \Delta y(t) \vert \leq \sqrt{\sum_{n \geq 1}\frac{\phi_n'(0)^2}{\lambda_n^2}}\sqrt{\sum_{n \geq 1} \lambda_n^2 \Delta w_n(t)^2} + \vert \Delta u(t) \vert$. In view of the stability estimate (\ref{eq: dirichlet boundary measurement - stab result}) provided by Theorem~\ref{thm: Case of a Neumann boundary measurement - stab}, we only need to study the term $\sum_{n \geq 1} \lambda_n^2 \Delta w_n(t)^2$. This can be done as in \cite[Proof of Theorem~2]{lhachemi2020pi}, yielding the claimed estimate (\ref{eq: neumann boundary measurement - reg result}).
\end{proof}

\section{Dirichlet measurement and Neumann regulation control}\label{sec: crossed configuration}
We now consider the reaction-diffusion system described by (\ref{eq: dirichlet boundary measurement - RD system - 1}-\ref{eq: dirichlet boundary measurement - RD system - 3}), still in the case $p \in \mathcal{C}^2([0,1])$, but this time with the boundary measurement $y_m(t)$ and the distinct and unmeasured output to-be-regulated $y_r(t)$ described by:
\begin{equation}\label{eq: crossed - RD system - measurements}
y_m(t) = z(t,0) , \quad y_r(t) = z_x(t,1) .
\end{equation}

\subsection{Control design}

Using the change of variable (\ref{eq: change of variable}), we obtain that (\ref{eq: homogeneous RD system - 1}-\ref{eq: homogeneous RD system - 3}) still hold while (\ref{eq: homogeneous RD system - 4}) is replaced by
\begin{subequations}\label{eq: crossed - measurement}
\begin{align}
& \tilde{y}_m(t) = w(t,0) = z(t,0) = y_m(t) , \\
& \tilde{y}_r(t) = w_x(t,1) = z_x(t,1) -2u(t) = y_r(t) -2u(t) .
\end{align}
\end{subequations}
Then, considering classical solutions, (\ref{eq: homogeneous RD system - spectral reduction - 1}-\ref{eq: homogeneous RD system - spectral reduction - 2}) is still valid while (\ref{eq: homogeneous RD system - spectral reduction - 3}) is replaced by
\begin{equation}\label{eq: crossed - measurement bis}
\tilde{y}_m(t) = \sum_{i \geq 1} \phi_i(0) w_i(t) , \quad \tilde{y}_r(t) = \sum_{i \geq 1} \phi_i'(1) w_i(t) .
\end{equation}
Based on similar motivations that the ones reported in Section~\ref{sec: Dirichlet measurement and regulation control}, we consider the integral component
\begin{align}\label{eq: crossed - integral component}
\dot{\xi}(t)
& = \sum_{n = 1}^{N_0} \phi_n'(1) \hat{w}_n(t) + \alpha_0 u(t) + \beta_0 v(t) - r(t) .
\end{align}
with
\begin{equation}\label{eq: crossed - def alpha_0 and beta_0}
\alpha_0 = 2 - \sum_{n \geq N_0 + 1} \frac{a_n \phi_n'(1)}{-\lambda_n + q_c} , \quad
\beta_0 = - \sum_{n \geq N_0 + 1} \frac{b_n \phi_n'(1)}{-\lambda_n + q_c}
\end{equation}
and where the observation dynamics, for $1 \leq n \leq N$, takes the form:
\begin{align}
\dot{\hat{w}}_n (t)
& = ( -\lambda_n + q_c ) \hat{w}_n(t) + a_n u(t) + b_n v(t) \nonumber \\
& \phantom{=}\; - l_n \left( \sum_{i=1}^N \phi_i(0) \hat{w}_i(t) - \alpha_1 u(t) - \tilde{y}_m(t) \right) \label{eq: crossed - observer dynamics - 1}
\end{align}
with
\begin{equation}\label{eq: crossed - def alpha_1}
\alpha_1 = \sum_{n \geq N +1} \dfrac{a_n \phi_n(0)}{-\lambda_n + q_c}
\end{equation}
and where $l_n \in\R$ are the observer gains. We set $l_n = 0$ for $N_0+1 \leq n \leq N$. Adopting now the same definitions as the ones used in Section~\ref{sec: Dirichlet measurement and regulation control} except that the matrix $A_1$, originally defined by (\ref{eq: def matrices A1 B1 and Br}), is now replaced by
$
A_1 = \begin{bmatrix} 0 & 0 & 0 \\ B_{0,a} & A_0 & 0 \\ \alpha_0 & C_r & 0 \end{bmatrix}
$
where $C_r = \begin{bmatrix} \phi_1'(1) & \ldots & \phi_{N_0}'(1) \end{bmatrix}$,  we infer that (\ref{eq: dynamics closed-loop system - finite dimensional part}) holds.

\begin{lemma}\label{lem: crossed - controllability lemma}
The pair $(A_0,C_0)$ is observable. If the unique solution of $(pf')' + (q_c - q) f =0$ with $f(1)=1$ and $f'(1)=0$ is such that $f'(0) \neq 0$, then the pair $(A_1,B_1)$ is controllable.
\end{lemma}

\begin{proof}
The observability of $(A_0,C_0)$ was assessed in Lemma~\ref{eq: Dirichlet measurement - Kalman condion}. From~\cite[Lem.~2]{lhachemi2020pi}, and because the pair (\ref{eq: kalman condition previous work}) is controllable, then $(A_1,B_1)$ is controllable if and only if the matrix
$
T = \begin{bmatrix} 0 & 0 & 1 \\ B_{0,a} & A_0 & B_{0,b} \\ \alpha_0 & C_r & \beta_0 \end{bmatrix}
$
is invertible. Let $\begin{bmatrix} u_e & w_{1,e} & \ldots & w_{N_0,e} & v_e \end{bmatrix}^\top \in\mathrm{ker}(T)$. We obtain that $v_e = 0$, $a_n u_e + (-\lambda_n +q_c) w_{n,e} = 0$ for all $1 \leq n \leq N_0$, and $\alpha_0 u_e + \sum_{n=1}^{N_0} \phi_n'(1) w_{n,e} = 0$. Defining for $n \geq N_0 +1$ the quantity $w_{n,e} = -\frac{a_n}{-\lambda_n+q_c}u_e$, we have $(-\lambda_n+q_c)w_{n,e} + a_n u_e = 0$ for all $n \geq 1$. Hence $(w_{n,e})_{n \geq 1} , (\lambda_n w_{n,e})_{n \geq 1} \in l^2(\N)$ ensuring that $w_e \triangleq \sum_{n \geq 1} w_{n,e} \phi_n \in D(\mathcal{A})$ and $\mathcal{A} w_e = \sum_{n \geq 1} \lambda_n w_{n,e} \phi_n$. This shows that $- \mathcal{A} w_e + q_c w_e + a u_e = 0$. Moreover, using (\ref{eq: crossed - def alpha_0 and beta_0}), we also have $0 = \alpha_0 u_e + \sum_{n=1}^{N_0} \phi_n'(1) w_{n,e} = 2 u_e + w_e'(1)$. From the two latter identities, we infer that $(p w_e')' + (q_c - q) w_e + a u_e = 0$, $w_e'(0) = w_e(1) = 0$, and $w_e'(1) + 2 u_e = 0$. Introducing the change of variable $z_e(x) = w_e(x) + x^2 u_e$, we deduce that $(p z_e')' + (q_c - q) z_e = 0$, $z_e'(0) = z_e'(1) = 0$, and $z_e(1) = u_e$. From our assumption, we infer that $u_e = 0$ and $z_e = 0$. Thus we have $w_e = z_e - x^2 u_e = 0$ hence $w_{n,e} = 0$ for all $n \geq 1$. We deduce that $\mathrm{ker}(T)=\{0\}$, showing that $(A_1,B_1)$ is controllable.
\end{proof}

We select $K \in\R^{1 \times (N_0 +2)}$ and $L \in\R^{N_0}$ so that $A_1 + B_1 K$ and $A_0 - L C_0$ are Hurwitz.

\subsection{Equilibrium condition and dynamics of deviations}

Proceeding as in Section~\ref{sec: Dirichlet measurement and regulation control}, we can characterize the equilibrium condition of the closed-loop system composed of the reaction-diffusion system (\ref{eq: dirichlet boundary measurement - RD system - 1}-\ref{eq: dirichlet boundary measurement - RD system - 3}) with (\ref{eq: crossed - RD system - measurements}), the auxiliary command input dynamics (\ref{eq: homogeneous RD system - abstract form - auxiliary input v}), the integral action (\ref{eq: crossed - integral component}), the observer dynamics (\ref{eq: crossed - observer dynamics - 1}), and the state-feedback (\ref{eq: v - state feedback}). In particular, setting $r(t) = r_e \in\R$, it can be shown that there exists a unique solution to :
\begin{subequations}\label{eq: crossed - equilibirum condition}
\begin{align}
0 & = (-\lambda_n+q_c) w_{n,e} + a_n u_e + b_n v_e = 0 , \quad n \geq 1 , \label{eq: crossed - equilibirum condition - 1} \\
0 & = v_e = K \hat{W}^{N_0}_{a,e} , \label{eq: crossed - equilibirum condition - 2} \\
0 & = \sum_{n=1}^{N_0} \phi_n'(1) \hat{w}_{n,e} + \alpha_0 u_e + \beta_0 v_e - r_e , \label{eq: crossed - equilibirum condition - 3} \\
0 & = (-\lambda_n+q_c) \hat{w}_{n,e} + a_n u_e + b_n v_e \nonumber \\
& \phantom{=}\; - l_n \left\{ \sum_{i=1}^{N} \phi_i(0) \hat{w}_{i,e} - \alpha_1 u_e - \tilde{y}_e \right\} , \; 1 \leq n \leq N_0 , \label{eq: crossed - equilibirum condition - 4} \\
0 & = (-\lambda_n+q_c) \hat{w}_{n,e} + a_n u_e + b_n v_e , \quad N_0 + 1 \leq n \leq N , \label{eq: crossed - equilibirum condition - 5} \\
\tilde{y}_{m,e} & = \sum_{n \geq 1} \phi_n(0) w_{n,e} , \quad
\tilde{y}_{r,e} = \sum_{n \geq 1} \phi_n'(1) w_{n,e} . \label{eq: crossed - equilibirum condition - 6}
\end{align}
\end{subequations}
Moreover we can define $w_e \triangleq \sum_{n \geq 1} w_{n,e} \phi_n \in D(\mathcal{A})$. Introducing the change of variable $z_e = w_e + x^2 u_e$, $z_e$ is a static solution of (\ref{eq: dirichlet boundary measurement - RD system - 1}-\ref{eq: dirichlet boundary measurement - RD system - 2}) associated with the constant control input $u(t) = u_e$. Denoting by $y_{r,e} \triangleq z_e'(1)$, we also infer that $y_{r,e} = r_e$, achieving the desired reference tracking. Consequently, we obtain the following dynamics of deviations:
\begin{subequations}\label{eq: crossed - dynamics of deviations}
\begin{align}
\Delta w(t,x) &= \Delta z(t,x) - x^2 \Delta u(t) , \\
\Delta \dot{X}(t) & = F \Delta X(t) + \mathcal{L} \Delta \zeta(t) - \mathcal{L}_r \Delta r(t) , \\
\Delta\zeta(t) & = \sum_{n \geq N+1} \phi_n(0) \Delta w_n(t) , \\
\Delta \dot{w}_n(t) & = (-\lambda_n + q_c) \Delta w_n(t) + a_n \Delta u(t) + b_n \Delta v(t) , \\
\Delta v(t) & = K \Delta \hat{W}_a^{N_0}(t) , \\
\Delta\tilde{y}_m(t) & = \Delta y_m(t) = \sum_{n \geq 1} \phi_n(0) \Delta w_n(t) , \\
\Delta\tilde{y}_r(t) & = \Delta y_r(t) - 2 \Delta u(t) = \sum_{n \geq 1} \phi_n'(1) \Delta w_n(t) .
\end{align}
\end{subequations}

\subsection{Stability analysis and regulation assessment}

The proof of the following theorem directly follows from the proofs reported in the previous sections.

\begin{theorem}\label{thm: crossed}
Under the assumption of Lemma~\ref{lem: crossed - controllability lemma},  the stability result stated by Theorem~\ref{thm: Case of a Dirichlet boundary measurement - stab} also applies to the closed-loop system composed of the plant (\ref{eq: dirichlet boundary measurement - RD system - 1}-\ref{eq: dirichlet boundary measurement - RD system - 3}) with (\ref{eq: crossed - RD system - measurements}), the integral actions (\ref{eq: homogeneous RD system - abstract form - auxiliary input v}) and (\ref{eq: crossed - integral component}), the observer dynamics (\ref{eq: crossed - observer dynamics - 1}), and the state feedback (\ref{eq: v - state feedback}). Moreover, for any $\eta \in [0,1)$, there exists $M_r > 0$ such that
\begin{align}
& \vert y_r(t) - r(t) \vert 
\leq M_r e^{-\delta t} \bigg( \vert \Delta u(0) \vert + \vert \Delta \xi(0) \vert + \sum_{n=1}^{N} \vert \Delta \hat{w}_n(0) \vert \nonumber \\
& + \Vert \Delta z_0 \Vert_{H^1} + \Vert \mathcal{A} \Delta w_0 \Vert_{L^2} \bigg) + M_r \sup_{\tau\in[0,t]} e^{-\eta\delta(t-\tau)} \vert \Delta r(\tau) \vert \label{eq: crossed - reg result}
\end{align}
for all $t \geq 0$ where $\Delta w_0 = \Delta z_0 - x^2 \Delta u(0)$.
\end{theorem}

\section{Numerical illustration}\label{sec: num}
We illustrate the result of Section~\ref{sec: crossed configuration} for Dirichlet measurement and Neumann regulation using a modal approximation that captures the 50 dominant modes of the reaction-diffusion PDE. We set $p=1$, $q=0$, and $q_c=3$ for which the open-loop plant is unstable. Selecting $\delta = 0.5$, we obtain $N_0 = 1$, the feedback gain $K = \begin{bmatrix} -10.4134 &  -11.3747 & 2.3100 \end{bmatrix}$, and the observer gain $L = 1.4373$. The conditions of Theorem~\ref{thm: crossed} are found feasible for $N = 3$. The time-domain evolution of the closed-loop system trajectories are depicted in Fig.~\ref{fig: sim}, confirming the theoretical predictions.

\begin{figure}
     \centering
     	\subfigure[State $z(t,x)$]{
		\includegraphics[width=3.25in]{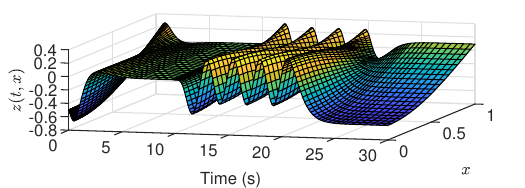}
		}
     	\subfigure[Observation error $e(t,x) = w(t,x) - \sum_{n=1}^{N} \hat{w}_n(t) \phi_n(x)$]{
		\includegraphics[width=3.25in]{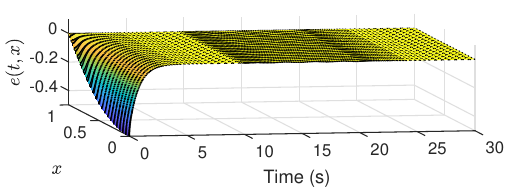}
		}
     	\subfigure[Command input $u(t) = z(t,1)$]{
		\includegraphics[width=3.25in]{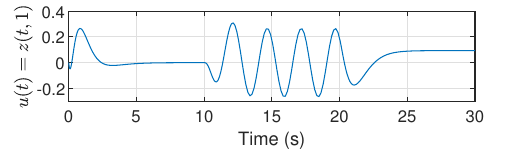}
		}
     	\subfigure[Regulated output $y_r(t) = z_x(t,1)$]{
		\includegraphics[width=3.25in]{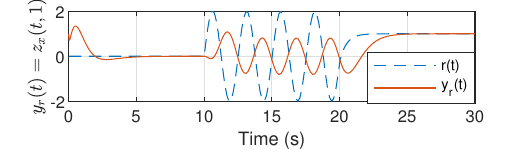}
		}
     \caption{Time evolution in closed-loop with Dirichlet boundary measurement $y_m(t) = z(t,0)$ and Neumann boundary regulation $y_r(t) = z_x(t,1)$ for the reaction-diffusion system (\ref{eq: dirichlet boundary measurement - RD system - 1}-\ref{eq: dirichlet boundary measurement - RD system - 3})}
     \label{fig: sim}
\end{figure}

\section{Conclusion}\label{sec: conclusion}
We proposed the design of a finite-dimensional observer-based PI controller to achieve the both output stabilization and regulation control of reaction-diffusion PDEs. Even if presented for a Dirichlet boundary control input, the presented results easily extend to Neumann/Robin boundary control (by modifying the change of variable formula to obtain an homogeneous PDE, giving different $a,b \in L^2(0,1)$). In-domain measurements can also be handled with the same approach provided the satisfaction of adequate observability conditions. While we have adopted in this paper an early lumping approach, future research directions for finite-dimensional PI regulation control of reaction-diffusion PDEs may be concerned with the study of late lumping approaches~\cite{auriol2019late} in the framework of backstepping control design for PDE-ODE cascades~\cite{tang2011state,wang2019output}.

\ifCLASSOPTIONcaptionsoff
  \newpage
\fi



\bibliographystyle{IEEEtranS}
\nocite{*}
\bibliography{IEEEabrv,mybibfile}

\begin{thebibliography}{10}
\providecommand{\url}[1]{#1}
\csname url@samestyle\endcsname
\providecommand{\newblock}{\relax}
\providecommand{\bibinfo}[2]{#2}
\providecommand{\BIBentrySTDinterwordspacing}{\spaceskip=0pt\relax}
\providecommand{\BIBentryALTinterwordstretchfactor}{4}
\providecommand{\BIBentryALTinterwordspacing}{\spaceskip=\fontdimen2\font plus
\BIBentryALTinterwordstretchfactor\fontdimen3\font minus
  \fontdimen4\font\relax}
\providecommand{\BIBforeignlanguage}[2]{{%
\expandafter\ifx\csname l@#1\endcsname\relax
\typeout{** WARNING: IEEEtranS.bst: No hyphenation pattern has been}%
\typeout{** loaded for the language `#1'. Using the pattern for}%
\typeout{** the default language instead.}%
\else
\language=\csname l@#1\endcsname
\fi
#2}}
\providecommand{\BIBdecl}{\relax}
\BIBdecl

\bibitem{auriol2019late}
J.~Auriol, K.~A. Morris, and F.~Di~Meglio, ``Late-lumping backstepping control
  of partial differential equations,'' \emph{Automatica}, vol. 100, pp.
  247--259, 2019.

\bibitem{balas1988finite}
M.~J. Balas, ``Finite-dimensional controllers for linear distributed parameter
  systems: exponential stability using residual mode filters,'' \emph{Journal
  of Mathematical Analysis and Applications}, vol. 133, no.~2, pp. 283--296,
  1988.

\bibitem{barreau2019practical}
M.~Barreau, F.~Gouaisbaut, and A.~Seuret, ``Practical stability analysis of a
  drilling pipe under friction with a {PI}-controller,'' \emph{IEEE
  Transactions on Control Systems Technology}, 2019.

\bibitem{coron2019pi}
J.-M. Coron and A.~Hayat, ``{PI} controllers for {1-D} nonlinear transport
  equation,'' \emph{IEEE Transactions on Automatic Control}, vol.~64, no.~11,
  pp. 4570--4582, 2019.

\bibitem{coron2004global}
J.-M. Coron and E.~Tr{\'e}lat, ``Global steady-state controllability of
  one-dimensional semilinear heat equations,'' \emph{SIAM Journal on Control
  and Optimization}, vol.~43, no.~2, pp. 549--569, 2004.

\bibitem{curtain1982finite}
R.~Curtain, ``Finite-dimensional compensator design for parabolic distributed
  systems with point sensors and boundary input,'' \emph{IEEE Transactions on
  Automatic Control}, vol.~27, no.~1, pp. 98--104, 1982.

\bibitem{curtain2020introduction}
R.~F. Curtain and H.~Zwart, \emph{An introduction to infinite-dimensional
  linear systems theory: A State-Space Approach}.\hskip 1em plus 0.5em minus
  0.4em\relax Springer, 2020, vol.~71.

\bibitem{dos2008boundary}
V.~Dos~Santos, G.~Bastin, J.-M. Coron, and B.~d'Andr{\'e}a Novel, ``Boundary
  control with integral action for hyperbolic systems of conservation laws:
  Stability and experiments,'' \emph{Automatica}, vol.~44, no.~5, pp.
  1310--1318, 2008.

\bibitem{harkort2011finite}
C.~Harkort and J.~Deutscher, ``Finite-dimensional observer-based control of
  linear distributed parameter systems using cascaded output observers,''
  \emph{International journal of control}, vol.~84, no.~1, pp. 107--122, 2011.

\bibitem{hespanha2018linear}
J.~P. Hespanha, \emph{Linear systems theory}.\hskip 1em plus 0.5em minus
  0.4em\relax Princeton university press, 2018.

\bibitem{katz2020constructive}
R.~Katz and E.~Fridman, ``Constructive method for finite-dimensional
  observer-based control of {1-D} parabolic {PDEs},'' \emph{Automatica}, vol.
  122, p. 109285, 2020.

\bibitem{krstic2008boundary}
M.~Krstic and A.~Smyshlyaev, \emph{Boundary control of PDEs: A course on
  backstepping designs}.\hskip 1em plus 0.5em minus 0.4em\relax SIAM, 2008.

\bibitem{lhachemi2020finite}
H.~Lhachemi and C.~Prieur, ``Finite-dimensional observer-based boundary
  stabilization of reaction-diffusion equations with either a {D}irichlet or
  {N}eumann boundary measurement,'' \emph{Automatica}, vol. 135, p. 109955,
  2022.

\bibitem{lhachemi2020pi}
H.~Lhachemi, C.~Prieur, and E.~Tr{\'e}lat, ``{PI} regulation of a
  reaction-diffusion equation with delayed boundary control,'' \emph{IEEE
  Transactions on Automatic Control}, vol.~66, no.~4, pp. 1573--1587, 2020.

\bibitem{lhachemi2021pi}
------, ``{PI} regulation control of a {1-D} semilinear wave equation,''
  \emph{SIAM Journal on Control and Optimization}, 2021, in press.

\bibitem{meurer2012control}
T.~Meurer, \emph{Control of higher--dimensional PDEs: Flatness and backstepping
  designs}.\hskip 1em plus 0.5em minus 0.4em\relax Springer Science \& Business
  Media, 2012.

\bibitem{morris2020}
K.~A. Morris, \emph{Controller Design for Distributed Parameter Systems}.\hskip
  1em plus 0.5em minus 0.4em\relax Springer, 2020.

\bibitem{orlov2017general}
Y.~Orlov, ``On general properties of eigenvalues and eigenfunctions of a
  {Sturm--Liouville} operator: comments on ''{ISS} with respect to boundary
  disturbances for {1-D} parabolic {PDEs}'','' \emph{IEEE Transactions on
  Automatic Control}, vol.~62, no.~11, pp. 5970--5973, 2017.

\bibitem{pazy2012semigroups}
A.~Pazy, \emph{Semigroups of linear operators and applications to partial
  differential equations}.\hskip 1em plus 0.5em minus 0.4em\relax Springer
  Science \& Business Media, 2012, vol.~44.

\bibitem{pohjolainen1982robust}
S.~Pohjolainen, ``Robust multivariable {PI}-controller for infinite dimensional
  systems,'' \emph{IEEE Transactions on Automatic Control}, vol.~27, no.~1, pp.
  17--30, 1982.

\bibitem{russell1978controllability}
D.~L. Russell, ``Controllability and stabilizability theory for linear partial
  differential equations: recent progress and open questions,'' \emph{{SIAM}
  Review}, vol.~20, no.~4, pp. 639--739, 1978.

\bibitem{sakawa1983feedback}
Y.~Sakawa, ``Feedback stabilization of linear diffusion systems,'' \emph{SIAM
  journal on control and optimization}, vol.~21, no.~5, pp. 667--676, 1983.

\bibitem{tang2011state}
S.~Tang and C.~Xie, ``State and output feedback boundary control for a coupled
  {PDE--ODE} system,'' \emph{Systems \& Control Letters}, vol.~60, no.~8, pp.
  540--545, 2011.

\bibitem{terrand2019adding}
A.~Terrand-Jeanne, V.~Andrieu, V.~D.~S. Martins, and C.~Xu, ``Adding integral
  action for open-loop exponentially stable semigroups and application to
  boundary control of {PDE} systems,'' \emph{IEEE Transactions on Automatic
  Control}, vol.~65, no.~11, pp. 4481--4492, 2020.

\bibitem{terrand2018regulation}
A.~Terrand-Jeanne, V.~Andrieu, M.~Tayakout-Fayolle, and V.~D.~S. Martins,
  ``Regulation of inhomogeneous drilling model with a {PI} controller,''
  \emph{IEEE Transactions on Automatic Control}, vol.~65, no.~1, pp. 58--71,
  2020.

\bibitem{trinh2017design}
N.-T. Trinh, V.~Andrieu, and C.-Z. Xu, ``Design of integral controllers for
  nonlinear systems governed by scalar hyperbolic partial differential
  equations,'' \emph{IEEE Transactions on Automatic Control}, vol.~62, no.~9,
  pp. 4527--4536, 2017.

\bibitem{wang2019output}
J.~Wang and M.~Krstic, ``Output feedback boundary control of a heat pde
  sandwiched between two {ODEs},'' \emph{IEEE Transactions on Automatic
  Control}, vol.~64, no.~11, pp. 4653--4660, 2019.

\bibitem{xu1995robust}
C.-Z. Xu and H.~Jerbi, ``A robust {PI}-controller for infinite-dimensional
  systems,'' \emph{International Journal of Control}, vol.~61, no.~1, pp.
  33--45, 1995.

\end{thebibliography}

\end{document}